\documentclass[final,a4paper, 10pt]{article}

\usepackage[a4paper,bottom=20mm,scale=0.768]{geometry}

\usepackage[english]{babel}

\usepackage[mathscr]{euscript}
\usepackage{color}
\usepackage{fancyhdr}
\usepackage{bm}
\usepackage{hyperref}
\usepackage{cite}
\usepackage{showkeys}
\usepackage{subfigure}
\usepackage{float}

\usepackage{amsfonts}
\usepackage{amsmath}
\usepackage{amssymb}
\usepackage{amscd}
\usepackage{amsthm}

\usepackage{accents}
\usepackage{graphicx}
\usepackage{array}

\newtheorem{theorem}{Theorem}

\newtheorem{lemma}[theorem]{Lemma}
\newtheorem{corollary}[theorem]{Corollary}

\newtheorem{remark}[theorem]{Remark}
\newtheorem{assumption}[theorem]{Assumption}

\newcommand*{\N}{\ensuremath{\mathbb{N}}}
\newcommand*{\Z}{\ensuremath{\mathbb{Z}}}
\newcommand*{\R}{\ensuremath{\mathbb{R}}}
\newcommand*{\C}{\ensuremath{\mathbb{C}}}
\renewcommand{\i}{\mathrm{i}}
\renewcommand{\epsilon}{{\varepsilon}}

\renewcommand{\d}[1]{\,\mathrm{d}#1 \,}

\renewcommand{\O}{\mathcal{O}}

\newcommand{\F}{\mathcal{F}} 
\newcommand{\p}{{\mathrm{per}}}
\newcommand{\0}{{0}} 

\newcommand{\X}{{\mathcal{X}}}

\renewcommand{\Re}{\mathrm{Re}\,}
\renewcommand{\Im}{\mathrm{Im}\,}

\newcommand{\x}{{\bm{\widetilde{x}}}}
\renewcommand{\j}{\bm{j}}
\newcommand{\bxi}{{\bm{\xi}}}

\newcommand{\bx}{{\bm{x}}}
\newcommand{\by}{{\bm{y}}}
\newcommand{\bz}{{\bm{z}}}
\newcommand{\balpha}{{\bm{\alpha}}}

\newlength{\dhatheight}

\usepackage{color}

\begin{document}

\sloppy

\title{Higher order convergence of perfectly matched layers in 3D bi-periodic surface scattering problems}
\author{
Ruming Zhang\thanks{Institute for Applied and Numerical mathematics, Karlsruhe Institute of Technolog (KIT), Karlsruhe, Germany; \texttt{ruming.zhang@kit.edu}}}
\date{}

\maketitle

\begin{abstract}
The perfectly matched layer (PML) is a very popular tool in the truncation of wave scattering in unbounded domains.  In \cite{Chand2009}, the author proposed a conjecture that for scattering problems with rough surfaces, the PML converges exponentially with respect to the PML parameter in any compact subset. In the author's previous paper \cite{Zhang2021b}, this result has been proved for periodic surfaces in two dimensional spaces, when the wave number is not a half integer. In this paper, we prove that the method has a high order convergence rate in the 3D bi-periodic surface scattering problems.  We extend the 2D results and prove that the exponential convergence still holds when the wavenumber is smaller than $0.5$. For lareger wavenumbers, although exponential convergence is no longer proved, we are able to prove that a higher order convergence for the PML method. 
\end{abstract}

\section{Introduction}

Since the PML method was proposed in \cite{Beren1994} by Berenger in 1994, it has been widely used in the simulation of wave propagation problems in unbounded domains. The main idea is to add an artificial absorbing layer outside the physical domain. The out going wave is abosrbed so fast within the layer that its Dirichlet/Neumann data is approximated by zero.  For the summary and drawbacks of the method we refer to \cite{John2010}. Since this method is not exact, a key question here is the convergence analysis. For bounded obstacles, we refer to \cite{Colli1998,Hohag2003,Lassa1998b,Chen2005a,Chen2007} for the error estimations.

For the case of unbounded obstacles, the problems become much more complicated and fewer results are available. In \cite{Chand2009}, for acoustic rough surface scattering problems, a linear convergence result was proved when the Neumann boundary condition was imposed. This also holds for the Dirichlet boundary condition with a similar proof. For Maxwell's equations, an exponential convergence was proved for absorbing wavenumbers, see \cite{Li2011}. For the special cases where the surfaces periodic and the incident fields are quasi-periodic, the exponential convergence is proved, and we refer to \cite{Chen2003,Zhou2018} for details.

In \cite{Chand2009}, the authors made a conjecture that although the convergence of the PML method was proved to be only linear, exponential convergence holds for any compact subdomains. Motivated by this conjecture, the author considered the 2D periodic surface scattering problems in \cite{Zhang2021b} and exponential convergence was proved except for the cases when the wavenumbers are half-integers.
When the wavenumber is a half-integer, since two singularities coincide, it is impossible to prove exponential convergence with the proposed method. What's worse, direct extension to 3D cases  holds only for small wavenumbers, i.e., when the wavenumber is smaller than $0.5$. For larger wavenumbers, the difficulty is exactly the same as the exceptional cases for 2D problems. In this paper, we will prove the high order convergence by a detailed study of coincided singularities. 

First, we introduce the mathematical model and important notations for this problem.  Let $\Gamma$ be a bi-periodic surface in $\R^3$ defined by
\[
\Gamma:=\big\{\big(t_1,t_2,\zeta(t_1,t_2)\big),\,t_1,\,t_2\in\R\big\},
\]
where $\zeta$ is a bi-periodic and bounded function. For simplicity, we assume that $\zeta$ is $2\pi$-periodic in both directions. The domain aboved $\Gamma$ is defined by 
\[
\Omega:=\big\{(t_1,t_2,t_3):\,t_1,\,t_2\in\R,\,t_3>\zeta(t_1,t_2)\big\}.
\]
We are considering the following problem:
\begin{equation}
 \label{eq:org}
 \Delta u+k^2 u=f\text{ in }\Omega;\quad u=0\text{ on }\Gamma,
\end{equation}
where $f\in L^2(\Omega)$ is compactly supported. 

\begin{remark}
    The result for \eqref{eq:org} can be easily extended to other boundary conditions, such as the impedance boundary condition of periodic layers, as long as the well posedness of the problem is guaranteed.
\end{remark}

Let $\Gamma_H:=\R^2\times\{H\}$ be a plane  lying above $\Gamma$, and the domain 
\[
\Omega_H:=\{x\in\R^3:\, \zeta(x_1,x_2)<x_3<H\}.
\]
For simplicity, let $f$ be compactly supported in $\Omega_H$. 
For the visualization  we refer to Figure \ref{fig:sample}.

\begin{figure}[h]
    \centering
    \includegraphics[width=0.5\textwidth]{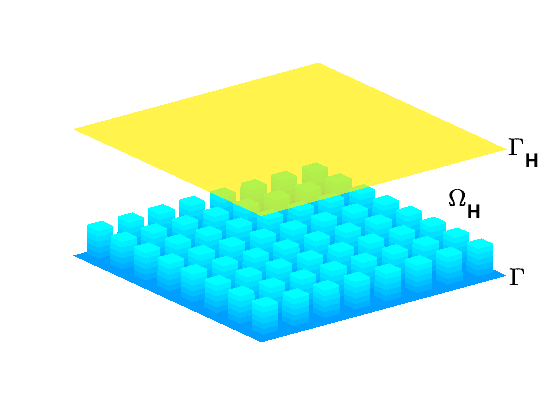}
    \caption{Domains and surfaces}
    \label{fig:sample}
\end{figure}

To guarantee that the solution $u$ is physical, we need the so-called upward propagating radiation condition (UPRC, see \cite{Chand2005}):
\[
u(\bx)=2\int_{\Gamma_H}\frac{\partial \Phi(\bx,\by)}{\partial x_3}u(\by)\d s(\by),\quad x_3>H,
\]
where $\Phi(\bx,\by)=\frac{\exp(\i k|\bx-\by|)}{4\pi|\bx-\by|}$ is the fundamental solution of the Helmholtz equation in $\R^3$. In \cite{Chand2005}, it was proved that the UPRC is equivalent to the following boundary condition:
\begin{equation}
 \label{eq:tbc}
 \frac{\partial u}{\partial x_3}=T^+ u=\i\int_{\R^2}\sqrt{k^2-|\bxi|^2}\widehat{u}(\bxi,H)e^{\i \x\cdot\bxi}\d\bxi,\quad \text{ on }\Gamma_H
\end{equation}
where $\x=(x_1,x_2)^\top$, the square root takes non-negative real and imaginar parts, and
\[
u(\x,H)=\int_{\R^2}\widehat{u}(\bxi,H)e^{\i \x\cdot\bxi}\d\bxi.
\]
With this boundary condition, the problem is reduced to the 3D bi-periodic domain $\Omega_H$ with finite height.

The well-posedness of the problem \eqref{eq:org}-\eqref{eq:tbc} has been proved in both standard and weighted Sobolev spaces in \cite{Chand2005,Chand2010}. We first introduce the weighted Sobolev space
\[
H^1_r(\Omega_H):=\left\{v\in H^1_{loc}(\Omega_H):\,(1+|\x|^2)^{r/2}v\in H^1(\Omega_H)\right\}.
\]
Given any compactly supported $f\in L^2(\Omega_H)$, the problem \eqref{eq:org}-\eqref{eq:tbc} has a unique weak solution in $H^1_r(\Omega_H)$ where $r\in(-1,1)$, and the solution $u$ depends continuously on $f$.

We apply the Floquet-Bloch transform to \eqref{eq:org}-\eqref{eq:tbc}, the solution $u$ is written as an integral of a family of periodic solutions, with respect to the Floquet parameter in a unit square. From the perturbation theory and Neumann series, the periodic solutions depend analytically on the Floquet parameter except for finite number of circles (called singular curves). For any point on the singular curves, there are square root singularities with respect to the Floquet parameter. It is easily proved that the PML solution converges exponentially when the Floquet parameters are away from the singular curves. But the convergence becomes slower when the point is closer to the singular curves. In particular, the convergence is linear on the curves. 

To deal with the slow convergence around the singular curves, we consider two cases separately. When the singular curve does not intersect with other curves, the 2D method is directly extended and the exponential convergence is proved. When multiple curves intersect at one point (fortunately we only have finite number of such points), we consider a small neighbourhood of this point. From a detailed study of the 
Neumann series around this point, we prove that the higher order convergence of the PML method. 

The rest of this paper is organized as follows. In Section 2, the Floquet-Bloch transform and its application 
 is introduced. In Section 3, the PML problems are described. We prove the convergence rate for a special case that $k<0.5$ in Section 4, and prove that for larger $k$'s  in Section 5.

The following are some important notations in this paper.

\begin{center}
    \begin{minipage}{\textwidth}
    \centering
    \begin{tabular}{|l |l|}
    \hline
    Notation &  Definition\\
    \hline\hline
       $\widetilde{\nabla}$  &$ \left(\frac{\partial }{\partial x_1},\,\frac{\partial}{\partial x_2}\right) $\\
       $\# S$ & The cardinality of the set $S$ \\
       $D(\balpha,r)$ & The open disk in $\R^2$ with center $\balpha$ and radius $r$\\
       $S(\balpha,r)$& The cirlce $\partial D(\balpha,r)$\\
       $B(k,\delta)$ & The open ball $\{s\in\C:\,|s-k|<\delta\}\subset\C$\\
       $B_-(k,\delta)$ & The open lower half ball $\{s\in\C:\,|s-k|<\delta,\,\Im(s)<0\}\subset\C$\\
       $C_-(k,\delta)$ & The half circle $\partial B_-(k,\delta)\setminus[k-\delta,k+\delta]+\i\{0\}$\\
       $C(k,\delta,[a,b])$& Partial circle of $\partial B(k,\delta)$ with angle in $[a,b]$\\
       $A(\balpha,r,\delta)$ & The full annulus with center $\balpha$, smaller radius $r-\delta$ and larger radius $r+\delta$\\
       $A(\balpha,r,\delta,[a,b])$ & The partial annulus of $A(\balpha,r,\delta)$ with angle in $[a,b]$\\
       $d(X,Y)$ & The distance between the sets $X$ and $Y$\\
       $\mathcal{O}(S,T,m,n)$\footnotemark & The sum of all permutations with $m$ times $S$ and $n$ times $T$\\
       \hline
    \end{tabular}   
    \end{minipage}
    \footnotetext[1]{Note that $\mathcal{O}(S,T,m,n)$ contains ${{{m+n}\choose{m}}}$ terms in total. In particular, when $S$ and $T$ are commutative, i.e., $ST=TS$, $\mathcal{O}(S,T,m,n)={{m+n}\choose{m}}S^m T^n$.}
\end{center}

\section{The Floquet-Bloch transform}

\subsection{Definition of the Floquet-Bloch transform}
First we introduce the following notataions. Let $\Omega_0:=\Omega\cap(-\pi,\pi)^2\times\R$, $\Omega_H^0:=\Omega_H\cap(-\pi,\pi)^2\times\R$, $\Gamma_0:=\Gamma\cap(-\pi,\pi)^2\times\R$ and $\Gamma_H^0:=\Gamma_H\cap(-\pi,\pi)^2\times\R$. For simplicity, we assume that $f$ is compactly supported in $\Omega_H^0$. Let the first Brilloiun zone be denoted as $W:=[-1/2,1/2]^2$. For the visualization of the periodicity cell we refer to Figure \ref{fig:cell}.

\begin{figure}[h]
    \centering
    \includegraphics[width=0.3\textwidth]{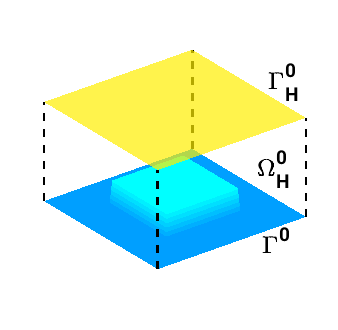}
    \caption{One periodicity cell.}
    \label{fig:cell}
\end{figure}

Define the Floquet-Bloch transform $\F$ for any smooth compactly supported function $\phi\in C_0^\infty(\Omega_H)$ by:
\[
(\F\phi)(\balpha,\bx)=\sum_{\j\in\Z^2}\phi(\x+2\pi \j,x_3)e^{-\i\balpha \cdot(\x+2\pi \j)},
\]
where $\balpha=(\alpha_1,\alpha_2)\in W$ is the Floquet parameter. 
From \cite{Lechl2015e,Lechl2016}, for any fixed $\balpha$, $(\F\phi)(\balpha,\cdot)$ is bi-periodic
with respect to $\x$ with the period $2\pi$, and $e^{\i\balpha\cdot\x}(\F\phi)(\balpha,\bx)$ is bi-periodic with respect to $\balpha$ with the period $1$. Moreover, the following equations hold:
\[
\left(\F\left(\frac{\partial\phi}{\partial x_\ell}\right)\right)(\balpha,\bx)=\left(\frac{\partial}{\partial x_\ell}-\i\alpha_\ell\right)(\F\phi)(\balpha,\bx),\,\ell=1,2,\quad \left(\F\left(\frac{\partial\phi}{\partial x_3}\right)\right)(\balpha,\bx)=\frac{\partial}{\partial x_3}(\F\phi)(\balpha,\bx).
\]

To introduce the important properties of the Floquet-Bloch transform $\F$, we first need the definitions of the following spaces. For details we refer to \cite{Lechl2016}. We begin with the space $H^\ell\left(W;H^s(\Omega_H^0)\right)$, where $\ell\in\N_0$ and $s\in\R$. This space contains all elements in the distribution space defined in $W\times\Omega_H^0$ with finite norm
\[
\|\psi\|_{H^\ell\left(W;H^s(\Omega_H^0)\right)}=\left[\sum_{\bm{\gamma}\in \N^2_0,\,|\bm{\gamma}|\leq\ell}\int_{W}\left\|\partial^{\bm{\gamma}}_\balpha \psi(\balpha,\cdot)\right\|^2_{H^s(\Omega_H^0)}\d\balpha\right]^{1/2}.
\]
Then for any $\0<\theta<1$, we define the space $H^{\ell+\theta}\left(W;H^s(\Omega_H^0)\right)$ by space interpolation:
\[
H^{\ell+\theta}\left(W;H^s(\Omega_H^0)\right)=\left[H^{\ell}\left(W;H^s(\Omega_H^0)\right),H^{\ell+1}\left(W;H^s(\Omega_H^0)\right)\right]_\theta.
\]
Thus $H^r\left(W;H^s(\Omega_H^0)\right)$ is well defined for all $r\geq 0$ and $s\in\R$. The dual space of $H^r\left(W;H^s(\Omega_H^0)\right)$ with respect to the inner product
\[
(\psi_1,\psi_2)=\int_{W}\left(\psi_1(\balpha,\cdot),\psi_2(\balpha,\cdot)\right)_{H^s(\Omega_H^0)}\d\balpha
\]
is denoted by $H^{-r}\left(W;H^s(\Omega_H^0)\right)$. Let $H^r\left(W;H^s_\p(\Omega_H^0)\right)$ be the subspace of $H^r\left(W;H^s(\Omega_H^0)\right)$, where all the elements are periodic with fixed $\balpha$. Now we are prepared to introduce the following properties of $\F$.

\begin{theorem}[Theorem 8, \cite{Lechl2016}]
    \label{th:fb}
    The Floquet-Bloch transform $\F$ is extended to an isomorphism between $H^s_r(\Omega_H)$ and $H^r(W;H^s_\p(\Omega_H^0))$ for any $s,r\in\R$. The inverse transform is defined as
    \[
\left(\F^{-1}\psi\right)(\x+2\pi \j,x_3)=\int_W \psi(\balpha,\bx)e^{\i\balpha\cdot(\x+2\pi \j)}\d\balpha,\quad \bx\in\Omega_H^0,\,\j\in\Z^2.
    \]
\end{theorem}

\subsection{Periodic problems and regularity}

Now we apply the Floquet-Bloch transform to the original problem \eqref{eq:org}. 
Let $w(\balpha,\bx):=\F u$. For any fixed Floquet parameter $\balpha$, from direct computation, $w(\balpha,\cdot)$ is bi-periodic and satisfies
\begin{equation}
 \label{eq:per}
 \Delta w(\balpha,\cdot)+2\i\balpha\cdot\widetilde{\nabla} w(\balpha,\cdot)+(k^2-|\balpha|^2)w(\balpha,\cdot)=e^{-\i\balpha\cdot \x}f\text{ in }\Omega_H^0;\quad w(\balpha,\cdot)=0\text{ on }\Gamma_0,
\end{equation}
with the boundary condition:
\begin{equation}
 \label{eq:tbc_per}
\frac{\partial w(\balpha,\cdot)}{\partial x_3}=T^+_\balpha w(\balpha,\cdot)=\i \sum_{\j\in\Z^2}\sqrt{k^2-|\balpha+\j|^2}\widehat{w}(\balpha,\j)e^{\i \j\cdot \x}\text{ on }\Gamma_H^0.
\end{equation}

Similar to the results in 2D cases in \cite{Kirsc1993}, the problem \eqref{eq:per}-\eqref{eq:tbc_per} has a unique solution in $H^1_\p(\Omega_H^0)$ for any fixed $\balpha\in W$, where $H^1_\p(\Omega_H^0)$ is the subspace of $H^1(\Omega_H^0)$ for bi-periodic functions. Thus $w(\balpha,\cdot)$ is well defined for any $\balpha\in W$, then the original solution $u$ is given by the inverse Floquet-Bloch transform:
\begin{equation}
 \label{eq:ifb}
 u(\bx)=\int_W e^{\i\balpha\cdot \x}w(\balpha,\bx)\d\balpha,\quad \bx\in\Omega_H.
\end{equation}

The variational form for \eqref{eq:per}-\eqref{eq:tbc_per} is given as follows. Find $w(\balpha,\cdot)\in H^1_\p(\Omega_H^0)$ such that
\begin{equation} \label{eq:per_var}
\begin{split}
   \int_{\Omega_H^0}\left[\nabla w(\balpha,\cdot)\cdot\nabla\overline{\phi}-2\i\balpha\cdot\widetilde{\nabla} w(\balpha,\cdot)\overline{\phi}-(k^2-|\balpha|^2)w(\balpha,\cdot)\overline{\phi}\right]\d \bx\\
    -4\pi^2 \i\sum_{\j\in\Z^2}\sqrt{k^2-|\balpha+\j|^2}\widehat{w}(\balpha,\j)\overline{\widehat{\phi}(\j)}=-\int_{\Omega_H^0}e^{-\i\balpha\cdot \x}f(\bx)\overline{\phi}(\bx)\d \bx.
\end{split}
\end{equation}
From Riesz representation theorem, we 
define the following operators and functions:
\begin{eqnarray*}
    &&\left<A(\balpha)\psi,\phi\right>:= \int_{\Omega_H^0}\left[\nabla \psi\cdot\nabla\overline{\phi}-2\i\left(\balpha\cdot\widetilde{\nabla} \psi\right)\overline{\phi}-(k^2-|\balpha|^2)\psi\overline{\phi}\right]\d \bx;\\
    && \left<B_{\j}\psi,\phi\right>=4\pi^2\i \widehat{\psi}(\j)\overline{\widehat{\phi}(\j)};\\
    && \left<G(\balpha,\cdot),\phi\right>=-\int_{\Omega_H^0}e^{-\i\balpha\cdot \x}f(\bx)\overline{\phi}(\bx)\d \bx;
\end{eqnarray*}
where $\phi,\psi\in H^1_\p(\Omega_H^0)$. The operators $A(\balpha)$ and $B_{\j}$ are bounded in $H^1_\p(\Omega_H^0)$. 
Since $|\balpha|^2=\alpha_1^2+\alpha_2^2$, all the operators and functions depend real analytically on $\balpha\in \R^2$. 
The variational equation \eqref{eq:per_var} is written as:
\begin{equation}
    \label{eq:per_operator}
    \left[A(\balpha)-\sum_{\j\in\Z^2}\sqrt{k^2-|\balpha+\j|^2}B_{\j}\right]w(\balpha,\cdot):=S(\balpha)w(\balpha,\cdot)=G(\balpha,\cdot).
\end{equation}
Since the problem is always well-posed for any $\balpha\in W$, there  is a constant $C>0$ such that
$\left\|S^{-1}(\balpha)\right\|\leq C$ holds for all $\balpha\in W$.

From \eqref{eq:per_operator}, when $\balpha$ is away from zeros of all the square roots, $S(\balpha)$ depends real analytically on $\balpha$ in $\R^2$. But when at least one square root is close to $0$, there is a square root singularity. Since those points are critical in the following analysis,  we need the following notations for simplicity. First, let 
\[
P:=\Big\{\balpha\in {W}:\, \exists\,\j\in \Z^2, \,\text{ such that }|\balpha+\j|=k\Big\}.
\]
For any $\balpha\in P$, define
\[
J(\balpha):=\Big\{\j\in\Z^2:\,|\balpha+\j|=k\Big\}.
\]
 We refer to Figure \ref{fig:singular} for different structures of $P$ with different wave number $k$'s.

\begin{remark}
    There are only finite number of $\balpha$'s such that $\# J(\balpha)>1$. In particular, when $k<0.5$, $P=S(\bm{0},k)$. Moreover, for all the $\balpha\in P$, $\# J(\balpha)=1$.
\end{remark}

\begin{figure}[h]
    \centering
    \includegraphics[width=0.32\textwidth]{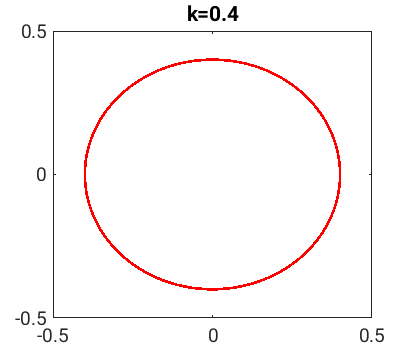}\includegraphics[width=0.32\textwidth]{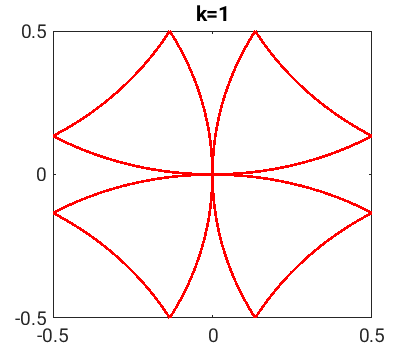}\includegraphics[width=0.32\textwidth]{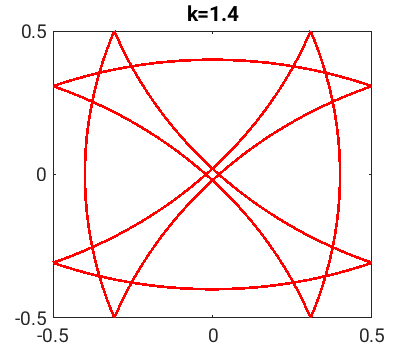}
    \caption{Structures of $P$ with different $k$'s.}
    \label{fig:singular}
\end{figure}

We consider a random point $\balpha_0\in P$, suppose $\# J(\balpha_0)=m\in\N_0$. Let $J(\balpha_0):=\{\j_1,\dots,\j_m\}$.  Since the square root function is continuous with real valued variables, there is a  sufficiently small $r_0>0$ such that for any  $\balpha\in D(\balpha_0,r_0)$, $\sqrt{k^2-|\balpha+\j|^2}\neq 0$ for all $\j\notin J(\balpha_0)$. In this case, let 
\[
\widetilde{A}(\balpha):=A(\balpha)-\sum_{\j\notin J(\balpha_0)}\sqrt{k^2-|\balpha+\j|^2}B_{\j},
\]
then $\widetilde{A}(\balpha)$ depends analytically on $\balpha\in D(\balpha_0,r_0)$. Moreover, when $r_0>0$ is sufficiently small, $\widetilde{A}(\balpha)=S(\balpha)-\sum_{\ell=1}^m \sqrt{k^2-|\balpha+\j_\ell|^2}B_{\j_\ell}$ is a small perturbation of $S(\balpha)$ thus $\widetilde{A}^{-1}(\balpha)$ exists and is uniformly bounded. Then
\[
    \left[\widetilde{A}(\balpha)+\sum_{\ell=1}^m\sqrt{k^2-|\balpha+\j_\ell|^2}B_{\j_\ell}\right]w(\balpha,\cdot)=G(\balpha,\cdot)
\]
is equivalent to
\[
\left[I-\sum_{\ell=1}^m\sqrt{k^2-|\balpha+\j_\ell|^2}\widetilde{B}_\ell(\balpha)\right]w(\balpha,\cdot)=\widetilde{G}(\balpha,\cdot)
\]
where $\widetilde{B}_\ell(\balpha)=-\widetilde{A}^{-1}(\balpha)B_{\j_\ell}$ and $\widetilde{G}(\balpha,\cdot)=-\widetilde{A}^{-1}(\balpha){G}(\balpha,\cdot)$. 
We can easily obtain the form of solution from the Neumann series:
\begin{equation}
\label{eq:neumann}
\begin{split}
w(\balpha,\cdot)&=\sum_{n=0}^\infty\left[\sum_{\ell=1}^m\sqrt{k^2-|\balpha+\j_\ell|^2}\widetilde{B}_\ell(\balpha)\right]^n\widetilde{G}(\balpha,\cdot)\\
&=\sum_{X\subset\{1,2,\dots,m\}}\left(\prod_{\ell\in X}\sqrt{k^2-|\balpha+\j_\ell|^2}\right)w_X(\balpha,\cdot),
\end{split}
\end{equation}
where $w_X(\balpha,\cdot)$ depends real analytically on $\balpha$ in  $D(\balpha_0,r_0)$. In particular, when $m=1$, $J(\balpha_0)=\{\j_1\}$ and
\begin{equation}
\label{eq:neu_1}
w(\balpha,\cdot)=w_0(\balpha,\cdot)+\sqrt{k^2-|\balpha+\j_1|^2}w_1(\balpha,\cdot),
\end{equation}
where $w_0(\balpha,\cdot)$, $w_1(\balpha,\cdot)$ depend real analytically on $\balpha$.

\section{Analytic extension and modification of the inverse Floquet-Bloch transform}
\label{sec:analytic}

In this section, we consider the analytic extension of $w(\balpha,\cdot)$ with respect to $\balpha$, when $\balpha$ is near the singular curves. With this extension, we also obtain an equivalent form of the inverse Floquet-Bloch transform \eqref{eq:ifb}. At first, we redefine the square root function "$\sqrt{\quad}$" to have the branch cut on negative imaginary axis:
\[
\sqrt{z}=\sqrt{|z|}\exp\left(\frac{\i t}{2}\right),\text{ where }t\in\left(-\frac{\pi}{2},\frac{3\pi}{2}\right].
\]

\subsection{Analytic extension }

For simplicity, we first assume for any $-\j\in\Z^2$, either $S(-\j,k)\cap W^{\mathrm{o}} \neq\emptyset$ or $S(-\j,k)\cap{W}=\emptyset$, where $W^{\mathrm{o}}=(-1/2,1/2)^2$ is the interior of $W$. This means that there are no curves which are tangential to $W$ from the exterior of it. So if a circle $S(-\j,k)$ does not pass through $W$, it has a positive distance with $\overline{W}$, defined as:
\[
\delta_W:=d\Big(S(-\j,k),W\Big).
\]
Actually, this assumption can be guaranteed by the translation of $W$. 

The number of singular circles that have non-empty intersection with $W$ is finite, denoted by
\[
S(-\j_\ell,k),\quad\ell=1,2,\dots,M.
\]
Note that when $k<0.5$, $M=1$ and $\j_1=\bm{0}$. It is more convenient to use the polar coordinate to study the analytical behaviour near the circles. For any $\bm{\ell}=\j_\ell$, let $\balpha$ be a point close to the circle $S(-\bm{\ell},k)$ and  set $\balpha:=(s\cos\phi,s\sin\phi)-\bm{\ell}$. 
With abuse of notation, we  replace the variable $\balpha\in\R^2$ in the functions and operators  with $(s,\phi)\in\R^2$ to denote the dependence of the new variables $s$ and $\phi$. Thus the variational problem \eqref{eq:per_var} is written as the new form:
\begin{equation} \label{eq:per_var_st}
\begin{split}
   &\int_{\Omega_H^0}\Big[\nabla w(s,\phi,\cdot)\cdot\nabla\overline{\phi}-2\i(s\cos\phi-\ell_1,s\sin\phi-\ell_2)\cdot\widetilde{\nabla} w(s,\phi,\cdot)\overline{\phi}\Big.
   \\&\qquad
   \Big.-(k^2-(s\cos\phi-\ell_1)^2-(s\sin\phi-\ell_2)^2)w(s,\phi,\cdot)\overline{\phi}\Big]\d \bx\\
    &-4\pi^2 \i\sum_{\j\in\Z^2}\sqrt{k^2-(s\cos\phi-\ell_1+j_2)^2-(s\sin\phi-\ell_2+j_2)^2}\widehat{w}(s,\phi,\j)\overline{\widehat{\phi}(\j)}\\=&-\int_{\Omega_H^0}e^{-\i(s\cos\phi-\ell_1,s\sin\phi-\ell_2)\cdot \x}f(\bx)\overline{\phi}(\bx)\d \bx.
\end{split}
\end{equation}
From the definition and the analytical dependence of $\balpha$ on $(s,\phi)$, $\widetilde{A}(s,\phi)$ and $G(s,\phi,\cdot)$ depend analytically on $(s,\phi)$. From the perturbation theory, the inverse operator $\widetilde{A}^{-1}(s,\phi)$ depends analytically in a small neighbourhood of $(s,\phi)$. 
Thus $\widetilde{B}(s,\phi)$ and $\widetilde{G}(s,\phi,\cdot)$ also depend analytically on $(s,\phi)$.

From the Neumann series \eqref{eq:neumann},  each $w_X(s,\phi,\cdot)$ involves finite number of operators $\widetilde{B}_\ell(s,\phi)$ and one element $\widetilde{G}(s,\phi,\cdot)$ and finite number of terms $\sqrt{k^2-|\balpha+\j_\ell|^2}^{2\ell_m}$ where $\ell_m\in\N$. Since
\[
\sqrt{k^2-|\balpha+\j_\ell|^2}^{2\ell_m}=\Big(k^2-(s\cos\phi-\ell_1+j_2)^2-(s\sin\phi-\ell_2+j_2)^2\Big)^{\ell_m}
\]
where $\ell_m$ is a non-negative integer, 
it depends analytically on $s$ and $\phi$. Then $w_X(s,\phi,\cdot)$ depends analytically on $s$ and $\phi$. Thus  there are two small positive numbers $\delta_0,\epsilon_0>0$ such that when $(s,\phi)\in B(k,\delta_0)\times B(\phi_0,\epsilon_0)\subset \C\times\C$, $w_X(s,\phi,\cdot)$ depends analytically on $(s,\phi)$. 

 For any point $(k,\phi_0)$ where $\phi_0\in[0,2\pi]$ and subset $X\subset\{1,2,\dots,m\}$, the function $w_X(s,\phi,\cdot)$ depends analytically in a neighourhood $B(k,\delta_0)\times B(\phi_0,\epsilon_0)$. Since $[0,2\pi]$ is compact and $[0,2\pi]\subset\cup_{\phi\in[0,2\pi]} B(\phi,\epsilon_\phi)$, from Heine-Borel Theorem, there is a finite cover $[0,2\pi]\subset\cup_{\ell=1}^N B(\phi_\ell,\epsilon_{\ell})$. Let $\delta:=\min\{\delta_{\phi_\ell}:\ell=1,2,\dots,N\}$, then $w_X(s,\phi,\cdot)$ depends analytically on $(s,\phi) $ in $B(k,\delta)\times \cup_{\ell=1}^N B(\phi_\ell,\epsilon_{\phi_\ell})\supset B(k,\delta)\times[0,2\pi]$. This implies that for any $\phi\in[0,2\pi]$,  $w_X(s,\phi,\cdot)$ depends analytically on $s\in B(k,\delta)$.  Note that here $w_X(s,\phi,\cdot)$ is defined piecewisely with respect to the angle $\phi$.  
 
Since the above result holds for any $\bm{\ell}=\j_\ell$, we can choose a sufficiently small $\delta>0$ such that $w_X(s,\phi,\cdot)$ depends analytically on $s\in B(k,\delta)\subset\C$ for any fixed $\phi\in[0,2\pi]$, in its own polar coordinate system. From the periodicity of the distribution of singular circles, this result holds uniformly in the neighbourhood for any circle $S(-\j,k)$ with  $\j\in\Z^2$.

\begin{remark}
   From now on, we also require that $\delta<\delta_W$ to avoid further complexities. 
\end{remark}

\subsection{Reformulation of the  inverse Floquet-Bloch transform}

For any circle $S(-\j_\ell,k)$ where $\ell=1,2,\dots,M$, we define the annulus $A(-\j_\ell,k,\delta)$, where the parameter $\delta>0$ is chosen in the previous subsection. The set
\[
R:=W\setminus\left[\cup_{\ell=1}^M A(-\j_\ell,k,\delta)\right].
\]
Then for any point in $R$, its distance to any singular circles is not less than $\delta$.

For the case $k<0.5$, there is only one circle $S(\bm{0},k)$ that has non-empty intersection with $W$. In this case, we also require that  $\delta<\min\{k,0.5-k\}$ to guarantee that $A(\bm{0},k,\delta)$ is a well defined annulus in the domain $W$. Then define 
\[
S:=A(\bm{0},k,\delta)\subset W,
\]
we easily split $W:=S\cup R$. But the case that $k\geq 0.5$ is much more complicated.

Suppose $\balpha_0$ is a point such that $\# J(\balpha_0)=m>1$. Let $J(\balpha_0)=\big\{\bm{\ell}_1,\dots,\bm{\ell}_m\big\}$, then all the circles $S(-\bm{\ell}_1,k),\dots, S(-\bm{\ell}_m,k)$ have a common point $\balpha_0$. Thus in the disk $D(\balpha_0,2\delta)$, the annuluses $A(-\bm{\ell}_1,k,\delta),\dots,A(-\bm{\ell}_m,k,\delta)$ have intersections with each other; outside this ball, each annulus does not have intersections with any other annuluses. For any $\bm{\ell}_t$ where $t=1,2,\dots,m$, we find two angles $\xi_t<\eta_t$ such that $A(-\bm{\ell}_t,k,\delta,[\xi_t,\eta_t])$ is the smallest such that $A(-\bm{\ell}_t,k,\delta)\cap D(\balpha_0,2\delta)\subset A(-\bm{\ell}_t,k,\delta,[\xi_t,\eta_t])$. Let
\[
Q(\balpha_0)=\cup_{t=1}^m A(-\bm{\ell}_t,k,\delta,[\xi_t,\eta_t]),
\]
then the area of $Q(\balpha_0)=O(\delta^2)$. Let
\[
Q:=\cup_{\balpha_0\in P:\,\# J(\balpha_0)>1} Q(\balpha_0),
\]
then it contains all the intersections with different singular circles. For the definition of the domain $Q(\balpha_0)$ we refer to Figure \ref{fig:bands}.

\begin{figure}[h]
    \centering
    \includegraphics[width=0.4\textwidth]{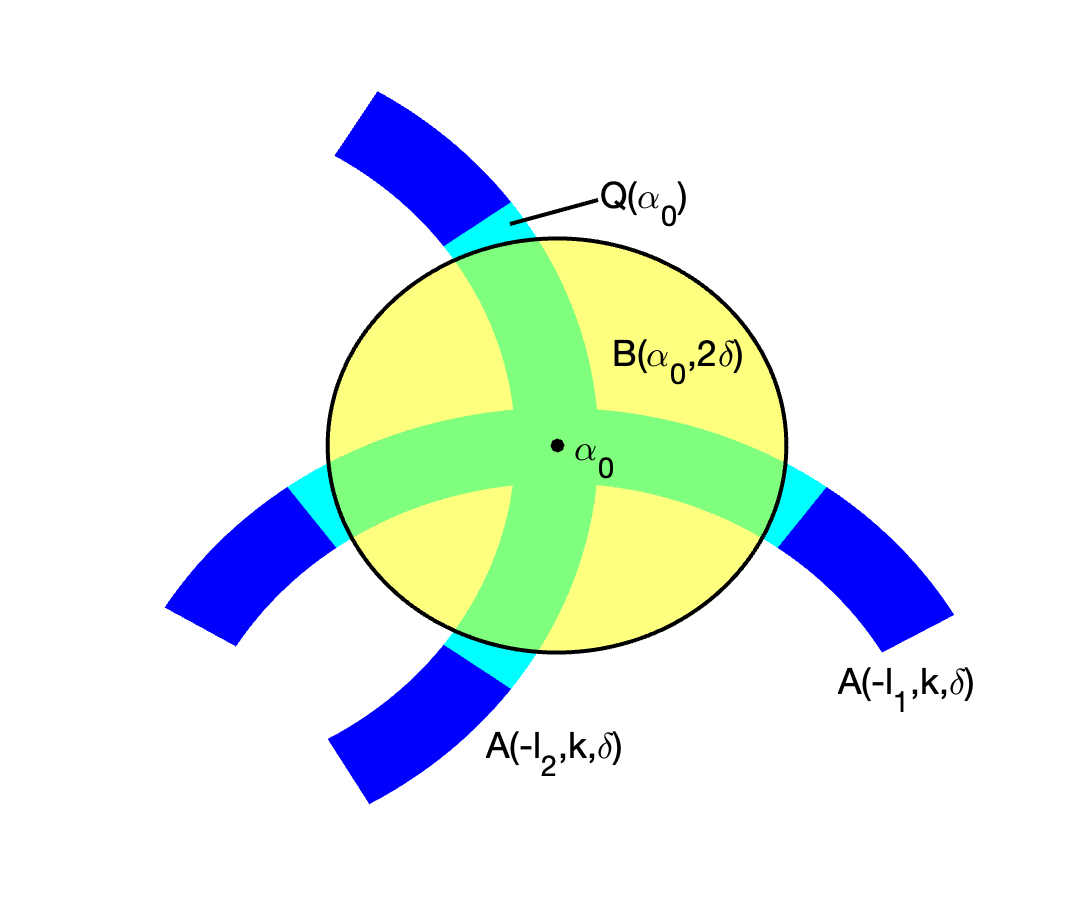}
    \caption{Intersection of two annulus. The light blue partial annuluses  $A(-\bm{\ell}_t,k,\delta,[\xi_t,\eta_t])$ are subsets of blue full annuluses $A(-\bm{\ell}_t,k,\delta)$. The yellow disk is $D(\balpha_0,2\delta)$. The union of light blue partial annuluses is the set $Q(\balpha_0)$.}
    \label{fig:bands}
\end{figure}

Now the inverse Floquet-Bloch transform is written as
\[
u(\bx)=\int_R e^{\i\balpha\cdot \x}w(\balpha,\bx)\d\balpha+\int_Q e^{\i\balpha\cdot \x}w(\balpha,\bx)\d\balpha+\int_{W\setminus(R\cup Q)} e^{\i\balpha\cdot \x}w(\balpha,\bx)\d\balpha.
\]
When $k<0.5$, $Q=\emptyset$ and the domain $W\setminus R$ is the full annulus $A(\bm{0},k,\delta)$. But when $k\geq 0.5$, $Q\neq\emptyset$ and the structure of $W\setminus (R\cup Q)$ becomes more complicated. 
Now we focus on the this domain. First we modify the set $W\setminus(R\cup Q)$. From the distribution of the singular circles (see Figure \ref{fig:curves}), 
\[
\cup_{\ell=1}^M \Big\{\balpha+\j_\ell:\,\balpha\in S(-\j_\ell,k)\cap W\Big\}=S(\bm{0},k).
\]
This result also holds for the annulus, since $\delta<\delta_W$:
\[
\cup_{\ell=1}^M \Big\{\balpha+\j_\ell:\,\balpha\in A(-\j_\ell,k,\delta)\cap W\Big\}=A(\bm{0},k,\delta).
\]
Thus we translate all the partial annulus $A(-\j_\ell,k,\delta)\cap W$ by $\balpha+\j_\ell$, then it becomes a subset of $A(\bm{0},k,\delta)$.

\begin{figure}[h]
    \centering
    \includegraphics[width=0.6\textwidth]{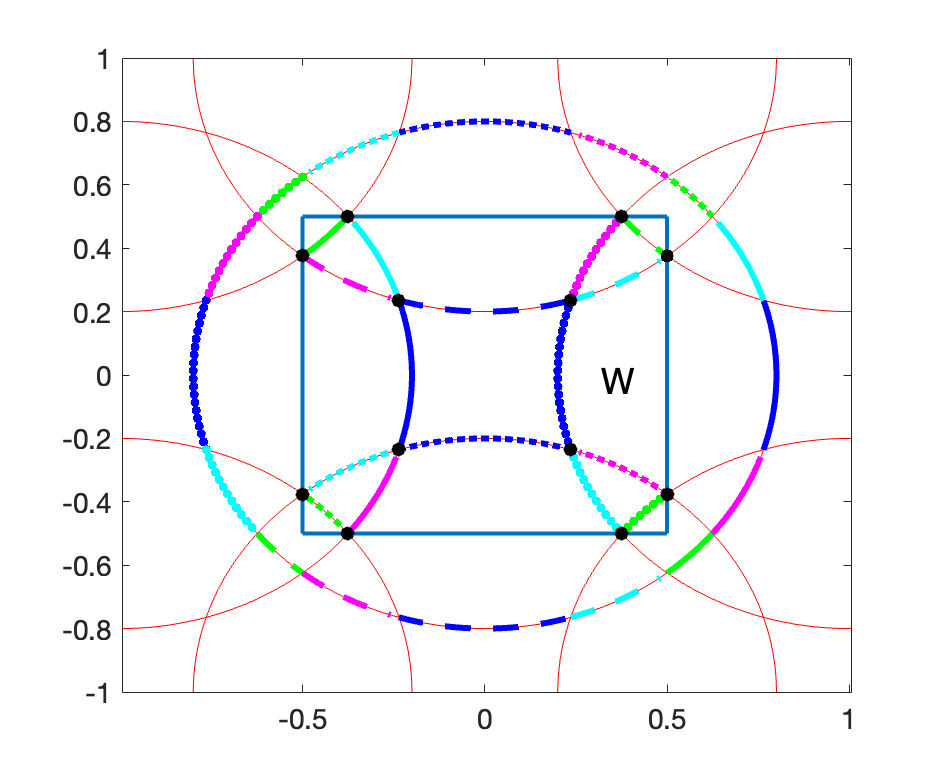}
    \caption{Distribution of the singular curves. The curves with the same color and line style are identical in the sence of periodic tranlation.}
    \label{fig:curves}
\end{figure}

Since the domain $Q$ is defined by the union of finite number of partial annulus, the set, 
\[
\cup_{\ell=1}^M \Big\{\balpha+\j_\ell:\,\balpha\in A(-\j_\ell,k,\delta)\cap W\setminus Q\Big\}
\]
is the union of several non-intersective  partial annuluses, denoted by $A(\bm{0},k,\delta,[\zeta_i,\psi_i])$ where $i=1,2,\dots,n$ for some positive integer $n$. Then from the periodicity of $e^{\i\balpha\cdot \x}w(\balpha,\bx)$ with respect to $\balpha$, finally we get
\[
\int_{W\setminus(R\cup Q)} e^{\i\balpha\cdot \x}w(\balpha,\bx)\d\balpha=\sum_{i=1}^n\int_{A(\bm{0},k,\delta,[\zeta_i,\psi_i])} e^{\i\balpha\cdot \x}w(\balpha,\bx)\d\balpha.
\]
So we define $S=\cup_{t=1}^n A(\bm{0},k,\delta,[\zeta_i,\psi_i])$, then 
\[
u(\bx)=\int_R e^{\i\balpha\cdot \x}w(\balpha,\bx)\d\balpha+\int_Q e^{\i\balpha\cdot \x}w(\balpha,\bx)\d\balpha+\int_{S} e^{\i\balpha\cdot \x}w(\balpha,\bx)\d\balpha.
\]

Finally we modify the integral on any partial annulus $A(\bm{0},k,\delta,[\zeta_i,\psi_i])$. With polar coordinate, 
\[
\int_{A(\bm{0},k,\delta,[\zeta_i,\psi_i])} e^{\i\balpha\cdot \x}w(\balpha,\bx)\d\balpha=\int_{\zeta_i}^{\psi_i}\int_{k-\delta}^{k+\delta}e^{\i s(\cos\phi,\sin\phi)\cdot \x}w(s,\phi,\bx) s\d s\d\phi.
\]
In the form \eqref{eq:neumann}, only the term $\sqrt{k^2-|\balpha+\bm{0}|^2}=\sqrt{k^2-s^2}$ has singularity, thus $w(s,\phi,\cdot)$ has the form of
\[
w(s,\phi,\cdot)=w_0(s,\phi,\cdot)+\sqrt{k^2-s^2} w_1(s,\phi,\cdot),
\]
where $w_0$ and $w_1$ depends analytically on $s\in B(k,\delta)$. Thus $w(s,\phi,\cdot)$ is extended analytically to $s\in B_-(k,\delta)$ due to the new definition of the square root function. For the visualization we refer to Figure \ref{fig:ball}. From Lemma 5 in \cite{Zhang2021b}, we get the modified formulation:
\[
\int_{A(\bm{0},k,\delta,[\zeta_i,\psi_i])} e^{\i\balpha\cdot \x}w(\balpha,\bx)\d\balpha=\int_{\zeta_i}^{\psi_i}\int_{C_-(k,\delta)}e^{\i s(\cos\phi,\sin\phi)\cdot \x}w(s,\phi,\bx) s\d s\d\phi.
\]
Then
\begin{equation}
    \label{eq:ift_mod}
 u(\bx)=\int_R e^{\i\balpha\cdot \x}w(\balpha,\bx)\d\balpha+\int_Q e^{\i\balpha\cdot \x}w(\balpha,\bx)\d\balpha+\sum_{i=1}^n \int_{\zeta_i}^{\psi_i}\int_{C_-(k,\delta)}e^{\i s(\cos\phi,\sin\phi)\cdot \x}w(s,\phi,\bx) s\d s\d\phi.
\end{equation}

\begin{figure}[h]
    \centering
    \includegraphics[width=0.35\textwidth]{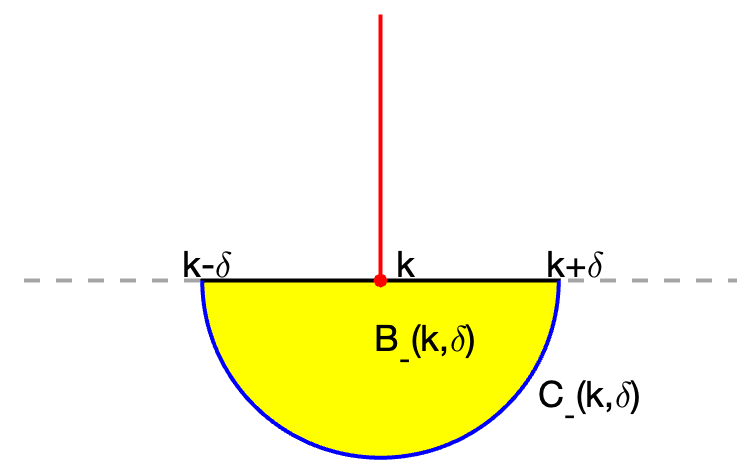}
    \caption{The domain $B_-(k,r)$ and curve $C_-(k,r)$. The red ray is the branch cut for the function $w(s,\phi,\cdot)$ with respect to $s$.}
    \label{fig:ball}
\end{figure}

\section{Approximation by perfectly matched layers}

We add a PML layer above $\Gamma_H$ with thickness $\lambda>0$. The layer is defined by a function $\kappa(x_3)=1+\rho\widehat{\kappa}(x_3)$. Here $\rho>0$ is a parameter, $\widehat{\kappa}(x_3)$ is a sufficiently smooth function:
\[
 \widehat{\kappa}(x_3)=\begin{cases}
     \X\left(\frac{x_3-H}{\lambda}\right)^m,\quad x_3\in[H,H+\lambda],\\
     0,\quad x_3<H,
 \end{cases}
\]
and $\X$ is fixed with positive real and imaginary parts, $m$ a positive integer to guarantee the smoothness of the function. Let
\[
 \sigma:=\int_H^{H+\lambda}\kappa(x_3)\d x_3=\lambda\left(1+\frac{\rho\X}{m+1}\right).
\]

We consider the PML solutions with two different boundary conditions, i.e., the Dirichlet and Neumann boundary condition. For simplicity, we denote the two solutions uniformly as $u^\sigma_I$ where $I=D,N$. The PML solution $u^\sigma_I$ is  described by the following equation with boundary condition on $\Gamma$:
\begin{equation}
 \label{eq:PML}
\widetilde{\nabla}\cdot\widetilde{\nabla}u^\sigma_I+\frac{1}{\kappa(x_3)}\frac{\partial}{\partial x_3}\left(\frac{1}{\kappa(x_3)}\frac{\partial u^\sigma_I}{\partial x_3}\right)+k^2 u^\sigma_I=f\text{ in }\Omega_{H+\lambda};\quad u^\sigma_I=0\text{ on }\Gamma,
\end{equation}
with  the Dirichlet boundary condition
\begin{equation}
   \label{eq:PML_D}  
    u^\sigma_D=0\text{ on }\Gamma_{H+\lambda},
\end{equation}
or the Neumann boundary condition
\begin{equation}
   \label{eq:PML_N}  
   \frac{\partial u^\sigma_N}{\partial x_3}=0\text{ on }\Gamma_{H+\lambda}.
\end{equation}

Note that in the physical domain $\Omega_H$, the PML solution $u^\sigma_I$ satisfies
\[
\Delta u^\sigma_I+k^2 u^\sigma_I=f\text{ in }\Omega_H;\quad u^\sigma_I=0\text{ on }\Gamma.
\]
From \cite{Chand2009}, the transparent boundary condition for $u^\sigma_D$ can be given as follows:
\begin{equation}
 \label{eq:tbc_PML_D}
 \frac{\partial u^\sigma_D}{\partial x_3}=T^{\sigma,D} u^\sigma_D=\i\int_{\R^2}\sqrt{k^2-|\bxi|^2}\coth\left(-\i\sqrt{k^2-|\bxi|^2}\sigma\right)\widehat{u}^\sigma_D(\bxi,H)e^{\i \x\cdot\bxi}\d\bxi.
\end{equation}
Similarly, the transparent boundary condition for $u^\sigma_N$ is given by:
\begin{equation}
 \label{eq:tbc_PML_N}
 \frac{\partial u^\sigma_N}{\partial x_3}=T^{\sigma,N} u^\sigma_N=\i\int_{\R^2}\sqrt{k^2-|\bxi|^2}\tanh\left(-\i\sqrt{k^2-|\bxi|^2}\sigma\right)\widehat{u}^\sigma_N(\bxi,H)e^{\i \x\cdot\bxi}\d\bxi.
\end{equation}
Similar to the convergence result in \cite{Chand2009}, when $|\sigma|$ is sufficiently large, the PML problems are uniquely solvable and 
\[
\|u^\sigma_D-u\|_{H^1(\Omega_H)},\,\|u^\sigma_N-u\|_{H^1(\Omega_H)}\leq C|\sigma|^{-1}.
\]

Now we apply the Floquet-Bloch transform to $u^\sigma_I$. Let $w^\sigma_I(\balpha,x):=\F u^\sigma_I$, then for fixed $\balpha$, $w^\sigma_I(\balpha,\cdot)$ is bi-periodic with respect to $\x$ and satisfies the cell problem:
\begin{equation}
 \label{eq:per_pml}
 \Delta w^\sigma_I(\balpha,\cdot)+2\i \balpha\cdot\widetilde{\nabla}w^\sigma_I(\balpha,\cdot)+(k^2-|\balpha|^2)w^\sigma_I(\balpha,\cdot)=e^{-\i\balpha\cdot\x}f\text{ in }\Omega_H^0;\quad w^\sigma_I(\balpha,\cdot)=0\text{ on }\Gamma_0
\end{equation}
with the transparent boundary condition:
\begin{equation}
 \label{eq:tbc_per_pml}
\frac{\partial w^\sigma_I(\balpha,\cdot)}{\partial x_3}=T^{\sigma,I}_\balpha w^\sigma_I(\balpha,\cdot)=\i \sum_{\j\in\Z}h_I(\balpha,\sigma,\j)\widehat{w}^\sigma_I(\balpha,\j)e^{\i \j\cdot \x}\text{ on }\Gamma_H^0,
\end{equation}
where $$h_D(\balpha,\sigma,\j):=\sqrt{k^2-|\balpha+\j|^2}\coth\left(-\i\sqrt{k^2-|\balpha+\j|^2}\sigma\right)$$ and $$h_N(\balpha,\sigma,\j):=\sqrt{k^2-|\balpha+\j|^2}\tanh\left(-\i\sqrt{k^2-|\balpha+\j|^2}\sigma\right).$$
Similar to \eqref{eq:per}-\eqref{eq:tbc_per}, the problem \eqref{eq:per_pml}-\eqref{eq:tbc_per_pml} is written as the variational problem
\begin{equation}
    \label{eq:per_pml_var}
    \begin{split}
   \int_{\Omega_H^0}\left[\nabla w^\sigma_I(\balpha,\cdot)\cdot\nabla\overline{\phi}-2\i\balpha\cdot\widetilde{\nabla} w^\sigma_I(\balpha,\cdot)\overline{\phi}-(k^2-|\balpha|^2)w^\sigma_I(\balpha,\cdot)\overline{\phi}\right]\d \bx\\
    -4\pi^2 \i\sum_{\j\in\Z^2}h_I(\balpha,\sigma,\j)\widehat{w}^\sigma_I(\balpha,\j)\overline{\widehat{\phi}(\j)}=-\int_{\Omega_H^0}e^{-\i\balpha\cdot \x}f(\bx)\overline{\phi}(\bx)\d \bx.
    \end{split}
\end{equation}
This results in the following operator equation:
\begin{equation}
\left[A(\balpha)-\sum_{\j\in\Z^2}h_I(\balpha,\sigma,\j)B_{\j}\right]w^\sigma_I(\balpha,\cdot):=S^\sigma_I(\balpha)w^\sigma_I(\balpha,\cdot)=G(\balpha,\cdot).
\end{equation}
From \cite{Chand2009}, $S^\sigma_I(\balpha)$ converges to $S(\balpha)$ uniformly with respect to $\alpha\in\overline{W}$. Thus when $|\sigma|>>1$, the problem \eqref{eq:per_pml_var} is uniquely solvable for any $\balpha\in W$. Then the PML solution is given by the inverse Floquet-Bloch transform: 
\begin{equation}
    \label{eq:ifb_pml}
 u^\sigma_I(\bx)=\int_W e^{\i\balpha\cdot \x}w^\sigma_I(\balpha,\bx)\d\balpha,\quad \bx\in\Omega_H.   
\end{equation}
With a similar but easier process as in Section \ref{sec:analytic}, the above formula is rewritten as follows:
\begin{equation}
    \label{eq:ift_pml_mod}
 u^\sigma_I(\bx)=\int_R e^{\i\balpha\cdot \x}w^\sigma_I(\balpha,\bx)\d\balpha+\int_Q e^{\i\balpha\cdot \x}w^\sigma_I(\balpha,\bx)\d\balpha+\sum_{i=1}^n \int_{\zeta_i}^{\psi_i}\int_{C_-(k,\delta)}e^{\i s(\cos\phi,\sin\phi)\cdot \x}w^\sigma_I(s,\phi,\bx) s\d s\d\phi.
\end{equation}

Since for sufficiently large $|\sigma|$, $S^\sigma_I(\balpha)$ is always invertible and $\left\|S^{-\sigma}_I(\balpha)\right\|\leq C$ by adjusting the constant $C$. Let $J(\balpha)=\{\j_1,\dots,\j_m\}$ and define $\widetilde{A}^\sigma_I(\balpha)=A(\balpha)+\sum_{\j\notin J(\balpha_0)}h_I(\balpha,\sigma,\j)B_{\j}$. Within the domain $\widetilde{P}$ defined in the previous section, the solution can  be written as the Neumann series:
\[
w^\sigma_I(\balpha,\cdot)=\sum_{n=0}^\infty\left[\sum_{\ell=1}^m h_I(\balpha,\sigma,\j)\widetilde{B}_{I,\ell}^\sigma(\balpha)\right]^n\widetilde{G}^\sigma_I(\balpha,\cdot),
\] where $\widetilde{B}_{I,\ell}^\sigma=-\widetilde{A}^{-\sigma}_I(\balpha)B_{\j_\ell}$ and $\widetilde{G}^\sigma_I(\balpha,\cdot)=-\widetilde{A}^{-\sigma}_I(\balpha)G(\balpha,\cdot)$. 
In particular, when $k<0.5$, since $P=\partial B(0,k)$,
\begin{equation}
\label{eq:small_w_sigma}
w^\sigma_I(\balpha,\cdot)=\sum_{n=0}^\infty\left[h_I(\balpha,\sigma,\j)\widetilde{B}_{I,0}^\sigma(\balpha)\right]^n\widetilde{G}^\sigma_I(\balpha,\cdot).
\end{equation}

\section{Convergence of the PML}

In this section, we will estimate the error between $u(\bx)$ and $u^\sigma_I(\bx)$. From \eqref{eq:ift_mod} and \eqref{eq:ift_pml_mod}, the difference between these two functions can be bounded by the following separated terms:
\begin{equation}
\label{eq:diff_plit}
\begin{split}
    \left\|u-u^\sigma_I\right\|_{H^1(M)}&\leq \left\|\int_R e^{\i\balpha\cdot\x}\left[w(\balpha,\cdot)-w^\sigma_I(\balpha,\cdot)\right]\d\balpha\right\|_{H^1(M)}\\&+\sum_{\balpha_0\in P:\,\# J(\balpha_0)>1}\left\|\int_{Q(\balpha_0)} e^{\i\balpha\cdot\x}\left[w(\balpha,\cdot)-w^\sigma_I(\balpha,\cdot)\right]\d\balpha\right\|_{H^1(M)}\\
    &+\sum_{i=1}^n\left\|\int_{\zeta_i}^{\psi_i}\int_{C_-(k,\delta)}e^{\i s(\cos\phi,\sin\phi)\cdot \x}\left[w(s,\phi,\cdot)-w^\sigma_I(s,\phi,\cdot)\right]s\d s\d\phi\right\|_{H^1(M)},
    \end{split}
\end{equation}
where $M\subset\Omega_H$ is any fixed compact subset of $\Omega_H$. For simplicity, let $\sigma=\sigma_1+\i\sigma_2=|\sigma|(\cos\theta+\i\sin\theta)$ for some fixed $\theta\in(0,\pi/2)$. 

\subsection{Estimation for the integral on $R$ }

In this section, we estimate the following integral:
\begin{equation}
    \label{eq:ir}
    I_R(\sigma,I):=\left\|\int_R e^{\i\balpha\cdot\x}\left[w(\balpha,\cdot)-w^\sigma_I(\balpha,\cdot)\right]\d\balpha\right\|_{H^1(M)}.
\end{equation}

With the definition that $R=W\setminus\left[\cup_{\ell=1}^M A(-\j_\ell,k,\delta)\right]$ and the fact that $\delta<\delta_W$, for any $\balpha\in R$, 
\[
d(\balpha,S(-\j,k))\geq \delta,\quad\text{ for all }\j\in\Z^2.
\]
To estimate $I_R(\sigma,I)$, we only need to compare the solutions of the periodic problem \eqref{eq:per_var} and the periodic PML problem \eqref{eq:per_pml_var}, when $\balpha\in R$. Since the only difference in the two problems lies in the DtN maps, we only need to estimate
\[
\left\|T^+_{\balpha}-T^{\sigma,I}_{\balpha}\right\|,\quad\text{ where }I=D,N.
\]
From the definitions of the DtN maps, we only need to estimate the following function for any $\bz=\balpha+\j$:
\begin{align*} 
q_D(\bz,\sigma):=h_D(\balpha,\sigma,\j)-\sqrt{k^2-|\balpha+\j|^2}&=\sqrt{k^2-|\bz|^2}\left(\coth\left(-\i\sqrt{k^2-|\bz|^2}\sigma\right)-1\right)\\&=\frac{2\sqrt{k^2-|\bz|^2}}{\exp\left(-\i\sqrt{k^2-|\bz|^2}\sigma\right)-1},
\end{align*}
or
\begin{align*} 
q_N(\bz,\sigma):=h_N(\balpha,\sigma,\j)-\sqrt{k^2-|\balpha+\j|^2}&=\sqrt{k^2-|\bz|^2}\left(\tanh\left(-\i\sqrt{k^2-|\bz|^2}\sigma\right)-1\right)\\&=-\frac{2\sqrt{k^2-|\bz|^2}}{\exp\left(-\i\sqrt{k^2-|\bz|^2}\sigma\right)+1}.
\end{align*}
The values of the two functions are estimated in the following lemma.

\begin{lemma}
    \label{lm:est1}
  The following estimation holds uniformly for $\bz=\balpha+\j$ where $\balpha\in R$ and $\j\in\Z^2$:
    \[
\left|q_I(\bz,\sigma)\right|\leq C\exp\left(-c\sqrt{k\delta} |\sigma|\right),
\]
where $c,C>0$ are independent of $|\sigma|$.
\end{lemma}

\begin{proof}
To prove the exponential decay of $q_D$ or $q_N$, we only need to prove the exponential increase of the function $\left|\exp\left(-\i\sqrt{k^2-|\bz|^2}\sigma\right)\right|$. 

Since for any $\j\in\Z^2$, $d(\balpha,S(-\j,k))\geq \delta$, either $|\balpha+\j|=|\bz|\geq k+\delta$ or $|\balpha+\j|=|\bz|\leq k-\delta$. When $|\bz|\geq k+\delta$, then $|\bz^2|-k^2=(|\bz|-k)(|\bz|+k)\geq 2k\delta$. Thus
\[
\left|\exp\left(-\i\sqrt{k^2-|\bz|^2}\sigma\right)\right|=\exp\left(\sqrt{|\bz|^2-k^2}\sigma_1\right)\geq\exp\left(\sqrt{2k\delta}|\sigma|\cos\theta\right). 
\]
When $|\bz|\leq k-\delta$, then $k^2-|\bz|^2=(k-|\bz|)(k+|\bz|)\geq k\delta$. Thus
\[
\left|\exp\left(-\i\sqrt{k^2-|\bz|^2}\sigma\right)\right|=\exp\left(\sqrt{k^2-|\bz|^2}\sigma_2\right)\geq\exp\left(\sqrt{k\delta}|\sigma|\sin\theta\right). 
\]
Since $\theta\in(0,\pi/2)$ is a fixed angle, $\sin\theta,\,\cos\theta>0$. Thus there is a constant $c>0$ such that
\[
\left|\exp\left(-\i\sqrt{k^2-|\bz|^2}\sigma\right)\right|\geq \exp\left(c\sqrt{k\delta}|\sigma|\right)
\]
holds uniformly for both cases. 
The proof is finished.
\end{proof}

Following Theorem 9 in \cite{Zhang2021b}, it is easily proved that $\|T^+_\balpha-T^{\sigma,I}_\balpha\|\leq C\left(c\sqrt{k\delta}|\sigma|\right)$, thus $\|w(\balpha,\cdot)-w^\sigma_I(\balpha,\cdot)\|_{H^1_\p(\Omega_H^0)}\leq C\left(c\sqrt{k\delta}|\sigma|\right).$ Then  it is easy to get the final result in this section.

\begin{theorem}
    \label{th:ir}
    When $\theta\in(0,\pi/2)$ is fixed, $I_R(\sigma,I)\leq C \exp\left(-c\sqrt{k\delta}|\sigma|\right)$ holds for sufficiently large $|\sigma|$ and $I=D,N$.
\end{theorem}

The following results are concluded as a corollary of Lemma \ref{lm:est1}.

\begin{corollary}
    \label{cr:est_operator}Suppose $\delta<\delta_W$. 
For any $\balpha\in W$,   the operator $\widetilde{A}^\sigma_I(\balpha)$ (where $I=D,N$) converges to $\widetilde{A}(\balpha)$ exponentially:
    \[
\left\|\widetilde{A}^\sigma_I(\balpha)-\widetilde{A}(\balpha)\right\|\leq C\exp\left(-c\sqrt{k\delta}|\sigma|\right).
    \]
    Then the following estimations are obtained directly:
    \[
\left\|\widetilde{B}^\sigma_{I,\ell}(\balpha)-\widetilde{B}_\ell(\balpha)\right\|,\,\left\|\widetilde{G}^\sigma_I(\balpha,\cdot)-\widetilde{G}(\balpha,\cdot)\right\|\leq C\exp\left(-c\sqrt{k\delta}|\sigma|\right).
    \]    
\end{corollary}

\subsection{Estimation for the integral on $S$}

From the formulas \eqref{eq:ift_mod} and \eqref{eq:ift_pml_mod}, the integral on the set $S$ is equivalent to the sum of integrals on the modified partial annulus. Thus we define
\begin{equation}
    \label{eq:is}I_S^i(\sigma,I)=\left\|\int_{\zeta_i}^{\psi_i}\int_{C_-(k,\delta)}e^{\i s(\cos\phi,\sin\phi)\cdot \x}\left[w(s,\phi,\cdot)-w^\sigma_I(s,\phi,\cdot)\right]s\d s\d\phi\right\|_{H^1(M)}.
\end{equation}
Due to the complex variable $s$, the notation $|\bz|^2$ is replaced by $\bz_1^2+\bz_2^2$ in this section.

We need the following assumption to prove the convergence result.

\begin{assumption}
    \label{asp}
    The angle $\theta$ satisfies
    \[
\theta\in\left(\frac{\pi}{8},\frac{\pi-\arctan 2}{2}\right).
    \]
\end{assumption}

From the polar coordinate, $\balpha=(\balpha_1,\balpha_2):=\bm{a}+\i\bm{b}$ satisfies
\[
\balpha_1= s\cos\phi,\,\balpha_2=s\sin\phi\]
where $s\in C_-(k,\delta)$ and $\phi$ lies in one interval $[\zeta_i,\psi_i]$. Then $\bm{a}=\Re(\balpha)\in A(\bm{0},k,\delta,[\zeta_i,\psi_i])$ and $\bm{b}=\Im(\balpha)$. 
From the definition of the partial annulus, it is obvious that for any $\j\in\Z^2\setminus\bm\{0\}$,
\[
d(\bm{a}, S(-\j,k))\geq \delta.
\]
Similar to the previous section, we still need to estimate the functions $q_I(\bz,\sigma)$ when $\bz=\balpha+\j$ where $\balpha$ is given as above. The estimation is given in the following lemma, which is mainly based on the proof from  Lemma 8, \cite{Zhang2021b}.

\begin{lemma}
    \label{lm:est2}Assumption \ref{asp} holds.
  The following estimation holds uniformly for $\bz=\balpha+\j$ where $\balpha=(s\cos\phi,s\sin\phi)$ with $s\in C_-(k,\delta)$, $\phi\in[a_t,b_t]$ and $\j\in\Z^2$:
    \[
\left|q_I(\bz,\sigma)\right|\leq C\exp\left(-c\sqrt{k\delta} |\sigma|\right).
\]
\end{lemma}

\begin{proof}
Similar to the proof of Lemma \ref{lm:est1}, we still need to prove the exponential increase of the function $\left|\exp\left(-\i\sqrt{k^2-s^2}\sigma\right)\right|$. Note that here $\bz_1^2+\bz_2^2=s^2$, then $\sqrt{k^2-z_1^2-z_2^2}=\sqrt{k^2-s^2}$. For simplicity, let $\sqrt{k^2-s^2}=d e^{\i\tau}$.
We need to consider the following two cases separately.

\noindent
i) When $\j=\bm{0}$. Let $s=k-\delta e^{\i t}$ where $t\in[0,\pi]$. 
When $j=0$, then
\[
k^2-s^2=2k \delta e^{\i t}-\delta^2 e^{2\i t}.
\]
From direct computation, 
\[
d\geq\sqrt{2k \delta-\delta^2}>\sqrt{k\delta}\text{ and }\tau\in\left[0,\frac{\pi}{2}\right].
\]
Since $\theta\in\left(0,\frac{\pi}{2}\right)$, there is a $\gamma_1>0$ such that $\sin(\tau+\theta)\geq \gamma_1$. Thus
\[
\left|\exp\left(-\i\sqrt{k^2-\bz^2}\sigma\right)\right|\geq\exp(\sqrt{k\delta}\gamma_1|\sigma|).
\]

\noindent
ii) When $\j\neq\bm{0}$. Since $d(\bm{a},S(-\j,k))\geq \delta$, either $|\bm{a}+\j|\geq k+\delta$ or $|\bm{a}+\j|\leq k-\delta$. Moreover, $|\bm{b}|\leq\delta$. From direct computation,
\[
k^2-\bz_1^2-\bz_2^2=k^2-|\bm{a}+\j|^2+|\bm{b}|^2-2\i(\bm{a}+\j)\cdot\bm{b}
\]
When $|\bm{a}+\j|\geq k+\delta$, 
\[
d\geq\sqrt{|\bm{a}+\j|^2-|\bm{b}|^2-k^2}\geq\sqrt{(k+\delta)^2-k^2-\delta^2}=\sqrt{2k\delta},
\]
and
\[
\left|\tan(2\tau)\right|=\left|\frac{2(\bm{a}+\j)\cdot\bm{b}}{|\bm{a}+\j|^2-|\bm{b}|^2-k^2}\right|\leq \frac{2\delta|\bm{a}+\j|}{\delta(k+|\bm{a}+\j|)-\delta^2}\leq 2.
\]
This implies that \[\tau\in\left[\frac{\pi-\arctan 2}{2},\frac{\pi+\arctan 2}{2}\right].\]

When $|\bm{a}+\j|\leq k-\delta$,
\[
d\geq\sqrt{k^2+b^2-|\bm{a}+\j|^2}\geq\sqrt{k^2-(k-\delta)^2}\geq \sqrt{k\delta},
\]
and
\[
\left|\tan(2\tau)\right|=\left|\frac{2(\bm{a}+\j)\cdot\bm{b}}{|\bm{b}|^2+k^2-|\bm{a}+\j|^2}\right|\leq \frac{2(k-\delta)\delta}{k^2-(k-\delta)^2}<1.
\]
This implies that
\[
\tau\in\left(-\frac{\pi}{8},\frac{\pi}{8}\right).
\]
For both cases, when Assumption \ref{asp} holds, there is a constant $\gamma_2>0$ such that $\sin(\tau+\theta)\geq\gamma_2$. Thus 
\[
\left|\exp\left(-\i\sqrt{k^2-\bz_1^2-\bz_2^2}\right)\right|\geq \exp\left(\sqrt{k\delta}\gamma_2|\sigma|\right).
\]

The proof is finished by combing the results from (i) and (ii).
\end{proof}

With above result, we can also prove that $\left\|T^+_{\balpha}-T^{\sigma,I}_{\balpha}\right\|\leq C\exp\left(-c\sqrt{k\delta}|\sigma|\right)$ holds uniformly for 
$\balpha=(s\cos\phi,s\sin\phi)$ where $s\in C_-(k,\delta)$ and $\phi\in[a_t,b_t]$. Thus we can also prove $\left\|w^\sigma_I(\balpha,\cdot)-w(\balpha,\cdot)\right\|_{H^1_\p(\Omega_H^0)}\leq C\exp\left(-c\sqrt{k\delta}|\sigma|\right)$.
Finally we arrive at the main result in this section.

\begin{theorem}
    \label{th:is}
    Assumption \ref{asp} is satisfied, then
    $I_S(\sigma,I)\leq C \exp\left(-c\sqrt{k\delta}|\sigma|\right)$ holds for sufficiently large $|\sigma|$ and $I=D,N$.
\end{theorem}

\subsection{Estimation for the integral on $Q$}

Now we consider the following value:
\begin{equation}
    \label{eq:iq}
    I_{Q(\balpha_0)}(\sigma,I):=\left\|\int_{Q(\balpha_0)} e^{\i\balpha\cdot\x}\left[w(\balpha,\cdot)-w^\sigma_I(\balpha,\cdot)\right]\d\balpha\right\|_{H^1(M)},
\end{equation}
where $\balpha_0$ is a point such that $\# J(\balpha_0)=m>1$. Assume that $J(\balpha_0)=\{\j_\ell:\,\ell=1,2,\dots,m\}$, then $Q(\balpha_0)$ is the union of $m$ partial annuluses:
\[
Q(\balpha_0)=\cup_{\ell=1}^m A(-\j_\ell,k,\delta,[\xi_\ell,\eta_\ell]).
\]
In this section, we only consider the case that $m=2$ as an example to simplify the notations. For $m>2$ everything is very similar.

Recall that from  the Neumann series \eqref{eq:neumann},
\[
w(\balpha,\cdot)=\sum_{m,n=0}^\infty\sqrt{k^2-|\balpha+\j_1|^2}^m\sqrt{k^2-|\balpha+\j_2|^2}^{n}\O\left(\widetilde{B}_{\j_1}(\balpha),\widetilde{B}_{\j_2}(\balpha),m,n\right)\widetilde{G}(\balpha,\cdot),
\]
which contains term with $\sqrt{k^2-|\balpha+\j_1|^2}^m\sqrt{k^2-|\balpha+\j_2|^2}^n$. At the same time, 
\[
w^\sigma_I(\balpha,\cdot)=\sum_{m,n=0}^\infty h_I^m(\balpha,\sigma,\j_1)h_I^{n}(\balpha,\sigma,\j_2)\O\left(\widetilde{B}_{I,\j_1}^\sigma(\alpha),\widetilde{B}_{I,\j_2}^\sigma(\balpha),m,n\right)\widetilde{G}^\sigma_I(\balpha,\cdot).
\]
Define the following functions with $\ell=1,2$:
\begin{eqnarray*}
    && v_\ell(\balpha,\cdot):=\sum_{n=0}^\infty\sqrt{k^2-|\balpha+\j_\ell|^2}^n \widetilde{B}_{\j_\ell}^n(\balpha)\widetilde{G}(\balpha,\cdot);\\
    && v^\sigma_{I,\ell}(\balpha,\cdot):=\sum_{n=0}^\infty h_I^n(\balpha,\sigma,\j_1)\left(\widetilde{B}_{I,\j_\ell}^\sigma(\balpha)\right)^n\widetilde{G}^\sigma_I(\balpha,\cdot).
\end{eqnarray*}
With these functions, we  define
\begin{align*}
    v_0(\balpha,\cdot)&=w(\balpha,\cdot)+\widetilde{G}(\balpha,\cdot)-v_1(\balpha,\cdot)-v_2(\balpha,\cdot)\\
    &=\sum_{m,n=1}^\infty\sqrt{k^2-|\balpha+\j_1|^2}^m\sqrt{k^2-|\balpha+\j_2|^2}^{n}\O\left(\widetilde{B}_{\j_1}(\balpha),\widetilde{B}_{\j_2}(\balpha),m,n\right)\widetilde{G}(\balpha,\cdot)
\end{align*}
and
\begin{align*}
    v^\sigma_{I,0}(\balpha,\cdot)&=w^\sigma_{I}(\balpha,\cdot)+\widetilde{G}^\sigma_{I}(\balpha,\cdot)-v^\sigma_{I,1}(\balpha,\cdot)-v^\sigma_{I,2}(\balpha,\cdot)\\
    &=\sum_{m,n=1}^\infty h^m(\balpha,\sigma,\j_1)h^n(\balpha,\sigma,\j_2)\O\left(\widetilde{B}^\sigma_{I,\j_1}(\balpha),\widetilde{B}^\sigma_{I,\j_2}(\balpha),m,n\right)\widetilde{G}^\sigma_I(\balpha,\cdot).
\end{align*}
From the definitions of the new functions, \eqref{eq:iq} can be bounded by:
\begin{align*}
    I_Q(\sigma,I)&\leq \left\|\int_{Q(\balpha_0)} e^{\i\balpha\cdot\x}\left[\widetilde{G}(\balpha,\cdot)-\widetilde{G}^\sigma_I(\balpha,\cdot)\right]\d\balpha\right\|_{H^1(M)}+\left\|\int_{Q(\balpha_0)} e^{\i\balpha\cdot\x}\left[v_0(\balpha,\cdot)-v^\sigma_{I,0}(\balpha,\cdot)\right]\d\balpha\right\|_{H^1(M)}\\&+\left\|\int_{Q(\balpha_0)} e^{\i\balpha\cdot\x}\left[v_1(\balpha,\cdot)-v^\sigma_{I,1}(\balpha,\cdot)\right]\d\balpha\right\|_{H^1(M)}+\left\|\int_{Q(\balpha_0)} e^{\i\balpha\cdot\x}\left[v_2(\balpha,\cdot)-v^\sigma_{I,2}(\balpha,\cdot)\right]\d\balpha\right\|_{H^1(M)}\\
    &:=(I)+(II)+(III)+(IV).
    \end{align*}
The exponential convergence of the first term is obtained directly from Corollary \ref{cr:est_operator}.  In the rest of this section, we first estimate terms (III) and (IV), and then deal with  term (II).

\subsubsection{Estimation of terms (III) and (IV)}

Take (III) as an example. From Neumann series, $v_1(\balpha,\cdot)$ solves the equation
\[
\left[\widetilde{A}(\balpha)-\sqrt{k^2-|\balpha+\j_1|^2}B_{\j_1}\right]v_1(\balpha,\cdot)=G(\balpha,\cdot).
\]
Similary, $v^\sigma_{I,1}(\balpha,\cdot)$ solves
\[
\left[\widetilde{A}^\sigma_I(\balpha)-h(\balpha,\sigma,\j_1) B^\sigma_{I,\j_1}\right]v^\sigma_{I,1}(\balpha,\cdot)=G(\balpha,\cdot).
\]
Recall that the domain $Q(\balpha_0)$ contains the intersection of two circles $S(-\j_1,k)$ and $S(-\j_2,k)$. From the definition of $\widetilde{A}(\balpha)$, it does not contain any singularity within the domain $Q(\balpha_0)$; while the square root only contains one singularity on the circle $S(-\j_1,k)$. We study the singularity of $v_1(\balpha,\cdot)$ with respect to $\balpha$. Since there is only one square root singularity in the operator, $v_1(\balpha,\cdot)$ has the form \eqref{eq:neu_1} in a small neighbourhood of $S(-\j_1,k)$, i.e.,
\[
v_1(\balpha,\cdot)=v_{0,1}(\balpha,\cdot)+\sqrt{k^2-|\balpha+\j_1|^2}v_{1,1}(\balpha,\cdot).
\]
Let $\balpha=s(\cos\phi,\sin\phi)-\j_1$ where $\phi\in[\xi_1,\eta_1]$ and $|s-k|<<1$.  Following the process in Section \ref{sec:analytic}, there is a constant $\delta_1>0$ such that $v_{0,1}(s,\phi,\cdot)$ and $v_{1,1}(s,\phi,\cdot)$ are extended analytically $(s,\phi)\in B(k,\delta_1)\times[\xi_1,\eta_1]$.  Then for each fixed $\phi\in[\xi_1,\eta_1]$, $v_1(s,\phi,\cdot)$ is extended analytically to $B_-(k,\delta_1)$ with respect to $s$.  Thus   this term is rewritten as
\begin{align*}
\int_{Q(\balpha_0)} e^{\i\balpha\cdot\x} v_1(\balpha,\cdot)\d\balpha=&\int_{Q(\balpha_0)\setminus A(-\j_1,k,\delta_1,[\xi_1,\eta_1])} e^{\i\balpha\cdot\x} v_1(\balpha,\cdot)\d\balpha\\&+\int_{\xi_1}^{\eta_1}\int_{C_-(k,\delta_1)}e^{\i s(\cos\phi,\sin\phi)\cdot \x}w(s,\phi,\bx) s\d s\d\phi.
\end{align*}
The above decomposition also holds for $v^\sigma_{I,1}$. 
Thus $(III)$ can be estimated be the following two terms
\begin{align*}
 (III)&\leq \left\|\int_{Q(\balpha_0)\setminus A(-\j_1,k,\delta_1,[\xi_1,\eta_1])} e^{\i\balpha\cdot\x}\left[v_1(\balpha,\cdot)-v^\sigma_{I,1}(\balpha,\cdot)\right]\d\balpha\right\|_{H^1(M)}\\&+\left\|\int_{\xi_1}^{\eta_1}\int_{C_-(k,\delta_1)}e^{\i s(\cos\phi,\sin\phi)\cdot \x}w(s,\phi,\bx) s\d s\d\phi\right\|_{H^1(M)}\\
 &:=(V)+(VI).
\end{align*}
The estimation of (V) follows Theorem \ref{th:ir}, and the estimation of (VI) follows Theorem \ref{th:is}. Thus there are two constants $c,C>0$ such that
\[
(III)\leq C\exp\left(-c\sqrt{k\delta_1}|\sigma|\right).
\]
Similarly, for the term (IV), we can also find a constant $\delta_2>0$ and $c,C>0$ such that
\[
(IV)\leq C\exp\left(-c\sqrt{k\delta_2}|\sigma|\right).
\]

From now on, we set $\delta:=\min\{\delta,\delta_1,\delta_2\}$, and let $\delta_1=\delta_2=\delta$. Moreover, we also let $\delta=|\sigma|^{-2\gamma}$ for a sufficiently large $|\sigma|$ and a fixed number $\gamma\in(0,1)$.

\subsubsection{Estimation of term (II)}
From the above Neumann series of $v_0(\balpha,\cdot)$ and $v^\sigma_{I,0}(\balpha,\cdot)$, we need the function
\begin{align*}
    g_D(m,n,\balpha,\sigma)&:=
h^m_D(\balpha,\sigma,\j_1)h^n_D(\balpha,\sigma,\j_2)-\sqrt{k^2-|\balpha+\j_1|^2}^m\sqrt{k^2-|\balpha+\j_2|^2}^n
    \\&=\sqrt{k^2-|\balpha+\j_1|^2}^m\sqrt{k^2-|\balpha+\j_2|^2}^n\\&\quad\cdot\left[\coth\left(-\i\sqrt{k^2-|\balpha+ \j_1|^2}\sigma\right)^m\coth\left(-\i\sqrt{k^2-|\balpha+\j_2|^2}\sigma\right)^n-1\right].
\end{align*}
for the  Dirichlet boundary condition, and 
 the function
\begin{align*}
    g_N(m,n,\balpha,\sigma)&:=
h^m_N(\balpha,\sigma,\j_1)h^n_N(\balpha,\sigma,\j_2)-\sqrt{k^2-|\balpha+\j_1|^2}^m\sqrt{k^2-|\balpha+\j_2|^2}^n
    \\&=\sqrt{k^2-|\balpha+\j_1|^2}^m\sqrt{k^2-|\balpha+\j_2|^2}^n\\&\quad\cdot\left[\tanh\left(-\i\sqrt{k^2-|\balpha+ \j_1|^2}\sigma\right)^m\tanh\left(-\i\sqrt{k^2-|\balpha+\j_2|^2}\sigma\right)^n-1\right].
\end{align*}
for the  Neumann boundary condition.

\vspace{0.5cm}

We first estimate the function $g_D(m,n,\balpha,\sigma)$. The estimation is based on the following lemma.

\begin{lemma}
    \label{lm:est_coth}
For fixed $\theta_1,\,\theta_2>0$, there is a constant $C>0$ such that
\[
\left|s\left[\coth\left(-\i s(\theta_1+\i\theta_2)\right)-1\right]\right|\leq C
\]
holds uniformly for any $s\in[0,+\infty)$ and $-\i s \in[0,+\infty)$.
\end{lemma}

\begin{proof}
Since 
\[
{s}\left[\coth\left(-\i s(\theta_1+\i\theta_2)\right)-1\right]=\frac{2s}{\exp\left[-2\i s(\theta_1+\i\theta_2)\right]-1}.
\]
We only need to show that $\exp\left[-2\i s(\theta_1+\i\theta_2)\right]-1$ is uniformly bounded from below.

First, from the Taylor series of the exponential function, $\lim_{s\rightarrow 0}r(s)=\frac{\i}{\theta_1+\i\theta_2}$. From the continuity of the function, there is a small $\epsilon>0$ such that $|r(s)|\leq 2$ holds for all $|s|\leq\epsilon$.

When $s$ is real and positive, and $s>\epsilon$,
\[
\left|\exp\left[-2\i s(\theta_1+\i\theta_2)\right]-1\right|\geq \exp(2\epsilon\theta_2)-1>0.
\]
When $s$ is purely imaginary and $-\i s>\epsilon$,
\[
\left|\exp\left[-2\i s(\theta_1+\i\theta_2)\right]-1\right|\geq \exp(2\epsilon\theta_1)-1>0.
\]
Thus when $|s|\geq\epsilon$, $\left|\exp\left[-2\i s(\theta_1+\i\theta_2)\right]-1\right|\geq \exp(2\epsilon\theta)-1>0$.
Combining these two parts, we conclude that the function is uniformly bounded. The proof is finished.

\end{proof}

Now we are prepared to estimate the function $g_D(m,n,\balpha,\sigma)$.

\begin{lemma}
    \label{lm:est_gd}
Suppose $|\delta|= |\sigma|^{-2\gamma}$ for some large $|\sigma|$.  There is a constant $C>0$ such that 
    \[
\left|g_D(m,n,\balpha,\sigma)\right|\leq C^{m+n}|\sigma|^{-\gamma(m+n)}.
\]
\end{lemma}

\begin{proof} 
Let $s=\sqrt{k^2-|\balpha+\j_1|^2}|\sigma|$ and $t=\sqrt{k^2-|\balpha+\j_2|^2}|\sigma|$. Since $|\delta|=|\sigma|^{-2\gamma}$, $|s|,\,|t|\leq C |\sigma|^{1-\gamma}$ for some constant $C>0$.
From Lemma \ref{lm:est_tanh}, there is a constant $C$, independent of $s$ and $t$ such that
\[
\left|s\coth\left(-\i s(\theta_1+\i\theta_2)\right)\right|\leq C(1+|s|),\,\left|t\coth\left(-\i t(\theta_1+\i\theta_2)\right)\right|\leq C(1+|t|).
\]
From direct computation,
\begin{align*}
\left|g_D(m,n,\balpha,\sigma)
\right||\sigma|^{m+n}
=&\Big|s^m t^n\left[\coth\left(-\i s(\theta_1+\i\theta_2)\right)^m\coth\left(-\i t(\theta_1+\i\theta_2)\right)^n-1\right]\Big|\\
\leq & C^m C^n(1+|s|)^m(1+|t|)^n+|s|^m|t|^n\\
\leq & C^{m+n}\left(1+C|\sigma|^{1-\gamma}\right)^{m+n}+C^{m+n}|\sigma|^{(1-\gamma)(m+n)},
\end{align*}
then there is a constant $C>0$ such that $\left|g_D(m,n,\balpha,\sigma)
\right||\sigma|^{m+n}\leq C^{m+n}|\sigma|^{(1-\gamma)(m+n)}$ when $|\sigma|>>1$.
Thus the proof is finished.
\end{proof}

\vspace{0.5cm}

Now we move on to  the function $g_N(m,n,\balpha,\sigma)$. 
First we need the following uniform boundedness for the 
hyper tangent function.

\begin{lemma}
\label{lm:est_tanh}
For fixed $\theta_1,\,\theta_2>0$, there is a constant $C>0$ such that
\[
\left|\tanh\left(-\i s(\theta_1+\i\theta_2)\right)\right|\leq C
\]
holds uniformly for any $s\in[0,+\infty)$ and $-\i s \in[0,+\infty)$.
\end{lemma}

\begin{proof}
Instead of $\tanh\left(-\i s(\theta_1+\i\theta_2)\right)$, we only need to check if $\tanh\left(-\i s(\theta_1+\i\theta_2)\right)-1$ is uniformly bounded. From definition,
\[
\tanh\left(-\i s(\theta_1+\i\theta_2)\right)-1=-\frac{2}{\exp\left(-2\i s(\theta_1+\i\theta_2)\right)+1}.
\]
Thus we only need to show that $$r_N(s):=\exp\left(-2\i s(\theta_1+\i\theta_2)\right)+1$$ is uniformly bounded from below. For details we refer to the proof of Lemma \ref{lm:est_coth}.
\end{proof}

Now we are prepared to estimate the function $g_N(m,n,\balpha,\sigma)$.

\begin{lemma}
    \label{lm:est_gn}
    Suppose $|\delta|= |\sigma|^{-2\gamma}$ for some large $|\sigma|$. Then there is a constant $C>0$ such that 
    \[
\left|g_N(m,n,\balpha,\sigma)\right|\leq C^{m+n}|\sigma|^{-\gamma(m+n)}.
\]
\end{lemma}

\begin{proof}We use the same notations as in the proof of Lemma \ref{lm:est_gd}.    From direct computation,
\begin{align*}
\left|g_N(m,n,\balpha,\sigma)
\right||\sigma|^{m+n}
=&\left|s^m t^n\left[\tanh\left(-\i s(\theta_1+\i\theta_2)\right)^m\tanh\left(-\i t(\theta_1+\i\theta_2)\right)^n-1\right]\right|\\
\leq & |s|^m|t|^n(C^{m+n}+1)\leq C^{m+n}|\sigma|^{(1-\gamma)(m+n)}.
\end{align*}
The proof is finished.
\end{proof}

\vspace{0.5cm}

From Lemma \ref{lm:est_gd} and \ref{lm:est_gn}, we can find a constant $C>0$ such that
   \[
\left|g_I(m,n,\balpha,\sigma)\right|\leq C^{m+n}|\sigma|^{-\gamma(m+n)}
\]
holds uniformly for any $\balpha\in Q(\balpha_0)$. 
Now we are prepared estimate the term $(II)$.

\begin{theorem}
    \label{th:conv_M}
    When Assumption \ref{asp} holds. For sufficiently large $|\sigma|$ and $\gamma\in(0,1)$, there is a constant $C>0$ such that 
 \[
(II)
\leq C|\sigma|^{-6\gamma}.
\]   
\end{theorem}

\begin{proof}

From the Neumann series, 
\begin{align*}
(v^\sigma_{I,0}-v_0)(\balpha,\cdot)&=\sum_{m,n=1}^\infty h_I^m(\balpha,\sigma,\j_1)h_I^{n}(\balpha,\sigma,\j_2)\O\left(\widetilde{B}^\sigma_{I,\j_1}(\balpha),\widetilde{B}^\sigma_{I,\j_2}(\balpha),m,n\right)\widetilde{G}^\sigma_I(\balpha,\cdot)\\
&-\sum_{m,n=1}^\infty\sqrt{k^2-|\balpha+\j_1|^2}^m \sqrt{k^2-|\balpha+\j_2|^2}^{n}\O\left(\widetilde{B}_{\j_1}(\balpha),\widetilde{B}_{\j_2}(\balpha),m,n\right)\widetilde{G}(\balpha,\cdot)\\&=\sum_{m,n=1}^\infty h_I^m(\balpha,\sigma,\j_1)h_I^{n}(\balpha,\sigma,\j_2)\left[\O\left(\widetilde{B}^\sigma_{I,\j_1}(\balpha),\widetilde{B}^\sigma_{I,\j_2}(\balpha),m,n\right)\widetilde{G}^\sigma_I(\balpha,\cdot)\right.\\&\hspace{5.4cm}\left.-\O\left(\widetilde{B}_{\j_1}(\balpha),\widetilde{B}_{\j_2}(\balpha),m,n\right)\widetilde{G}(\balpha,\cdot)\right]\\
&+\sum_{m,n=1}^\infty g_I(m,n,\balpha,\sigma) \O\left(\widetilde{B}_{\j_1}(\balpha),\widetilde{B}_{\j_2}(\balpha),m,n\right)\widetilde{G}(\balpha,\cdot),
\end{align*}

For the estimation of $\left\|\widetilde{B}^\sigma_{I,\j_\ell}(\balpha)-\widetilde{B}_{\j_\ell}(\balpha)\right\|$
and  $\left\|\widetilde{G}^\sigma_I(\balpha,\cdot)-\widetilde{G}(\balpha,\cdot)\right\|$, we refer to Corollary \ref{cr:est_operator}. Thus we only need to focus on the second term.
Let the $(m,n)$-th term be denoted by
\begin{align*}
 I_{m,n}=g_I(m,n,\balpha,\sigma)\O\left(\widetilde{B}_{\j_1}(\balpha),\widetilde{B}_{\j_2}(\balpha),m,n\right)\widetilde{G}(\balpha,\cdot).   
\end{align*}
Let the norm of $\widetilde{B}_{\j_1}(\balpha)$, $\widetilde{B}_{\j_2}(\balpha)$ be uniformly bounded by a positive constant $C$, then
\[
|I_{m,n}|\leq C^{m+n}{{m+n}\choose{m}}\int_{Q(\balpha_0)}\left|g(m,n,\balpha,\sigma)\right|\d\balpha\leq C^{m+n}{{m+n}\choose{m}}|\sigma|^{-(m+n+4)\gamma}.
\]

Since 
\[
\sum_{m=1}^\infty\sum_{n=1}^\infty C^{m+n}{{m+n}\choose{m}}|\sigma|^{-(m+n+4)\gamma}\leq \sum_{n=2}(2C)^{n}|\sigma|^{-(n+4)\gamma},
\]
there is a constant $C>0$ such that
\[
\left|\sum_{m=1}^\infty\sum_{n=1}^\infty I_{m,n}\right|\leq C|\sigma|^{-6\gamma}
\]
holds uniformly for sufficiently large $|\sigma|$. 
The proof is finished.
\end{proof}

Combining the estimations on (I), (II), (III) and (IV), we get the main result in this subsection.
\begin{theorem}
    \label{th:iq}
    When $\theta\in(0,\pi/2)$ is fixed, $\balpha_0\in W$ satisfies $\#J(\balpha_0)>1$. Then for any fixed $\gamma\in(0,1)$, 
    $I_{Q(\balpha_0)}(\sigma,I)\leq C |\sigma|^{-6\gamma}$ holds for sufficiently large $|\sigma|$ and $I=D,N$.
\end{theorem}

\subsection{Convergence result for the PML method}

From the previous subsections, each component in the right hand side of \eqref{eq:diff_plit} has been well studied. Thus we only need to combine these results to get the convergence result.  In the domains $R$ and $S$, exponential convergence has been proved, with a similar technique from the 2D cases (see \cite{Zhang2021b}). However, near the intersection $Q(\balpha_0)$ between multiple annluses, only algebraic convergence, i.e, $|\sigma|^{-6\gamma}$. Fortunately we only have finite number of such intersection. For a special case that $k<0.5$, any $\balpha_0\in P$ has the property that $\# J(\balpha_0)=1$. In this case $W=R\cup S$, thus we can still prove the exponential convergence.
We conclude  the total convergence for $u^\sigma$ to $u$ in the following theorem.

\begin{theorem}
\label{th:conv_u}
Suppose Assumption \ref{asp} holds. Then 
\begin{itemize}
    \item When $k<0.5$, there are two positive constants $C$ and $c$ such that
    \[
\|u^\sigma_I-u\|_{H^1(\Omega_H^0)}\leq C\exp\left(-c|\sigma|\right).
    \]
    \item When $k\geq 0.5$, for any fixed parameter $\gamma\in(0,1)$,     there is a positive constant $C$  such that
   \[
\|u^\sigma_I-u\|_{H^1(\Omega_H^0)}\leq C|\sigma|^{-6\gamma}.
    \]
\end{itemize}
\end{theorem}

\begin{remark}
    The result is also extendable to the two dimensional cases, when $k$ is a half integer. The only difference is in the area of $Q(\balpha_0)$, which is only $O(|\sigma|^{-2\gamma})$ since it is one dimensional. So in this case, we can prove that the convergence is $O(|\sigma|^{-4\gamma})$, which is slower than the 3D bi-periodic case.
\end{remark}

In this paper, we prove that the PML method converges exponentially for small wave numbers but algebraically at a higher order for more general cases. Although the conjecture in \cite{Chand2009} is not proved, the convergence result is already improved significantly from this paper. Since the numerical computation for this problem is also challenging, we will discuss that and show numerical results in a following paper.

\section*{Acknowledgements.} 

This research was funded by the Deutsche Forschungsgemeinschaft (DFG, German Research Foundation) -- Project-ID 258734477 -- SFB 1173.

\bibliographystyle{plain}
\providecommand{\noopsort}[1]{}

\end{document}